\documentclass[11pt]{article}
 \usepackage[T1]{fontenc}
 \usepackage{chngcntr}
\usepackage{amsthm}
\usepackage{amssymb}
\usepackage{amsmath}
\usepackage{comment}
\usepackage{thm-restate}
\usepackage{latexsym}
\usepackage{graphicx,enumerate,color}
\usepackage{url} 
\usepackage[affil-it]{authblk}
\usepackage{hyperref}
\usepackage[noabbrev,capitalise]{cleveref}

\usepackage{color}
\usepackage[normalem]{ulem}

\usepackage[margin=1in]{geometry}

\theoremstyle{plain}
\newtheorem{thm}{Theorem}[section]
\newtheorem{lem}[thm]{Lemma}

\newtheorem{conj}[thm]{Conjecture}
\newtheorem{que}[thm]{Question}

{\noindent \emph{Proof.} {}{#1}{}}{\hfill
	$\Diamond$\vspace{1em}}

\theoremstyle{plain} 
\newcommand{\thistheoremname}{}
\newtheorem{genericthm}[section]{\thistheoremname}

\theoremstyle{definition}

\newcounter{counter}
\newcommand{\less}{\setminus}
 
 \def\dfn#1{{\sl #1}}

\title{The minimum degree of  $(K_s, K_t)$-co-critical graphs}
\author{Ivan Casas-Rocha\thanks{Current address: Department of Mathematics,   University of Minnesota. E-mail: {\tt casas061@umn.edu}} \hskip 1cm  Benjamin Snyder\thanks{Current address: Department of Statistics, Texas A\&M University. E-mail: {\tt b.snyder@tamu.edu}}\hskip 1cm   Zi-Xia Song\thanks{Supported by  NSF grant    DMS-2153945. E-mail: {\tt  Zixia.Song@ucf.edu}}}
 
  \affil{  \small {Department  of Mathematics, University of Central Florida, Orlando, FL 32816, USA}}

\begin{document}
\maketitle
\begin{abstract}
   Given   graphs $G, H_1, H_2$, we write  $G \rightarrow ({H}_1, H_2)$ if  every \{red, blue\}-coloring of the edges of $G$ contains  a  red copy of $H_1$ or a blue copy of $H_2$. A non-complete graph $G$ is $(H_1, H_2)$-co-critical if $G  \nrightarrow ({H}_1, H_2)$ and  $G+e\rightarrow ({H}_1, H_2)$ for every edge $e$ in the complement of $G$.  The notion of co-critical graphs was initiated by    Ne$\check{s}$et$\check{r}$il in 1986. 
 Galluccio,    Simonovits and   Simonyi in  1992 proved that every $(K_3, K_3)$-co-critical graph on $n\ge6$ vertices has minimum degree at least four, and the bound is sharp for all $n\ge 6$. In this paper, we first extend the aforementioned result to all $(K_s, K_t)$-co-critical graphs by showing that every $(K_s, K_t)$-co-critical graph  has minimum degree at least $2t+s-5$, where $t\ge s\ge 3$. We then prove that every $(K_3, K_4)$-co-critical graph on $n\ge9$ vertices has minimum degree at least seven, and the bound is sharp for all $n\ge 9$.   This answers   a question of the third author in the positive for the  case $s=3$ and $t=4$. 
\end{abstract}

\baselineskip 16pt

\section{Introduction}

 In this paper we consider graphs that are finite, simple and undirected.     For a graph $G$, we  use $V(G)$ to denote the vertex set, $E(G)$ the edge set,   $|G|$ the number of vertices in $G$,    $N(x)$ the neighborhood of vertex $x$ in $G$, $\delta(G)$ the minimum degree, $\Delta(G)$ the maximum degree,   and $\overline{G}$ the complement of $G$. If $S, T \subseteq V(G)$ are disjoint, we say that $S$ is \dfn{complete to} $T$ if every vertex in $S$ is adjacent to all vertices in $T$; and $S$ is \dfn{anti-complete to} $T$ if no vertex in $S$ is adjacent to any vertex in $T$.  We denote by $S\less T$  the set $S- T$,  and  $G \less S$ the subgraph obtained from $G$ by deleting all vertices in $S$.   The subgraph of $G$ induced by $S$, denoted $G[S]$, is the graph obtained from $G$ by deleting all vertices in $ V(G)\less S$.  For convenience, we use   $S\less v  $  to denote $S\less  \{v\}$, and say $v$ is complete to $S$ (resp.  anti-complete to $S$) when $T =\{v\}$. We use $G+e$ to denote the graph obtained from $G$ by adding the new edge $e$, where $e$ is an edge in $\overline{G}$.   Given a graph $H$,   $G$ is \emph{$H$-free} if $G$ does not contain $H$ as a subgraph;  and $G$ is \emph{$H$-saturated} if $G$ is $H$-free but $G+e$ is not $H$-free for every $e$ in $\overline{G}$.       
   \medskip

For any positive integer $k$, we write  $[k]$ for the set $\{1,2, \ldots, k\}$. A \dfn{$k$-edge coloring} of a graph $G$ is a function $\tau:E(G)\to [k]$.  We think of the set  $[k]$ as a set of colors, and we may identify a member of $[k]$ as a color, say, color $k$ is blue. Given an integer $k \ge 1 $ and graphs $G$, ${H}_1, \ldots, {H}_k$,   let   $\tau : E(G) \to [k]$  be a  $k$-edge coloring of $G$  with color classes  $E_1, \ldots, E_k$.  For each $\ell\in[k]$, we use $G_\ell$   to denote the spanning subgraph of $G$  with edge set   $E_\ell$.
We define  $\tau$ to be a  \emph{critical coloring} of $G$ if $G$ has no monochromatic $H_\ell$ for each $\ell\in[k]$, that is,    $G_\ell$ is $H_\ell$-free.  For each vertex $v$ in $G$,   we use  $N_\ell(v)$ to denote the   neighborhood of $v$ in $G_\ell$. We simply write  $N_b(v)$, $E_b$ and $G_b$  if the color $\ell$ is blue; and $N_r(v)$, $E_r$ and $G_r$  if the color $\ell$ is red.  
We write \dfn{$G \rightarrow ({H}_1, \ldots, {H}_k)$} if  $G$ admits no critical coloring; that is, every  $k$-edge coloring of   $G$ contains a monochromatic  ${H}_i$ in color $i$ for some $i\in[k]$. The classical \dfn{Ramsey number}  $r({H}_1, \dots, {H}_k)$ is the minimum positive integer $n$ such that $K_n \rightarrow ({H}_1, \dots, {H}_k)$. \medskip

Following~\cite{Galluccio1992,Nesetril1986}, a non-complete graph $G$ is $(H_1, \dots, H_k)$-\dfn{co-critical} if $G \not\rightarrow ({H}_1,\dots, {H}_k)$, but $G + e \rightarrow ({H}_1,\dots, {H}_k)$ for every edge $e$ in $\overline{G}$. The notion of co-critical graphs was initiated by Ne$\check{s}$et$\check{r}$il~\cite{Nesetril1986} in 1986   when he asked the following question regarding   $(K_3, K_3)$-co-critical graphs: 
\begin{quote}
 Are there   infinitely  many \dfn{minimal} co-critical graphs, i.e.,  co-critical graphs which lose this property when any vertex is deleted? Is $K_6^-$ the only one? 
\end{quote}
This was answered in the positive by Galluccio, Simonovits and Simonyi~\cite{Galluccio1992}.  They constructed infinitely many minimal $(K_3, K_3)$-co-critical graphs that are  $K_5$-free.  Szab\'o~\cite{Szabo1996} then constructed   infinitely many  nearly regular $(K_3, K_3)$-co-critical graphs with low maximum degree.     In the same paper they also studied   the minimum  degree of $(K_3, K_3)$-co-critical graphs, and made an observation on the chromatic number of  $(K_{t_1}, \ldots,  K_{t_k})$-co-critical graphs, where $t_1\ge2, \ldots, t_k\ge2$ are integers.

\begin{thm}[Galluccio,  Simonovits and Simonyi~\cite{Galluccio1992}]\label{t:K3K3}
Every $(K_3, K_3)$-co-critical graph on $n\ge6$ vertices  has minimum degree at least four. The bound is sharp for all $n\ge 6$.   

\end{thm}
\begin{lem}[Galluccio,  Simonovits and Simonyi~\cite{Galluccio1992}]\label{l:chi}
If  $G$ is a   $(K_{t_1}, \ldots,  K_{t_k})$-co-critical graph, then    $\chi(G)\ge r-1$,  and the equality holds  only when  $G$ is a complete $(r-1)$-partite graph, where $\chi(G)$ denotes the chromatic number of $G$ and $r=r(K_{t_1}, \ldots,  K_{t_k})$.  

\end{lem}

It is worth noting that \cref{l:chi} is not true in general. It is simple to check that  the complete bipartite graph $K_{2t-2, 2t-2}$ is $(K_{1,t}, K_{1,t})$-co-critical  and  $r(K_{1,t}, K_{1,t})\ge 2t-1>3$ for all $t>2$.   The {\dfn{join}} $G+H$  of two vertex disjoint graphs $G$ and $H$ is the graph having vertex set $V(G) \cup V(H)$  and edge set $E(G) \cup E(H) \cup \{xy\mid  x\in V(G),  y\in V(H)\}$.    Hanson and Toft~\cite{Hanson1987}  observed in 1987 that for all $n \ge r:=r(K_{t_1},\dots,K_{t_k})$,  the graph $K_{r-2}+ \overline K_{n-r+2}$ is $(K_{t_1},\dots, K_{t_k})$-co-critical with $(r - 2)(n - r + 2) + \binom{r - 2}{2}$ edges; they made the following  conjecture.

\begin{conj}[Hanson and Toft~\cite{Hanson1987}]\label{HTC}  
If $G$ is a $(K_{t_1},\dots, K_{t_k})$-co-critical graph on $n$ vertices, then 
\begin{align*}
|E(G)|\ge (r - 2)(n - r + 2) + \binom{r - 2}{2}, 
\end{align*}
where $r = r(K_{t_1}, \dots, K_{t_k})$.   The bound is best possible for all $n\ge r$.
\end{conj}
\medskip

 It was shown in \cite{Chen2011} that every $(K_3,K_3)$-co-critical graph on $n\ge 56$ vertices has at least $4n-10$ edges. This settles the first non-trivial case of Conjecture~\ref{HTC} for sufficiently large $n$.    Conjecture~\ref{HTC} remains wide open. 
 We refer the reader to a recent paper by Zhang and   the third author~\cite{SongZhang21} for further background on $(H_1, \dots, H_k)$-co-critical graphs, and to~\cite{ P3, Pk,c4star, DSY22,survey21, Ferrara2014, RolekSong18,SongZhang21}   for recent work on minimizing the number of edges in  $(H_1, \ldots, H_k)$-co-critical graphs. 
 It is worth noting that  the graph $K_{r-2}+ \overline K_{n-r+2}$ has minimum degree $r-2$. With the support of \cref{t:K3K3}, the third author  recently raised the following question.

\begin{que}[Song]\label{question}
Is it true that  every $(K_{t_1}, \ldots,  K_{t_k})$-co-critical graph has minimum degree at least $r(K_{t_1},\dots,K_{t_k})-2$?
\end{que}

The purpose  of this paper is  to   study the minimum  degree of $(K_{t_1}, \ldots, K_{t_k})$-co-critical graphs. We need to introduce more definitions. A \dfn{clique} in a graph $G$ is a set of pairwise adjacent vertices; and a \dfn{stable  set} is a set of pairwise non-adjacent vertices. A \dfn{$t$-clique} is a clique of order $t$. 
Let  $\tau: E(G)\rightarrow [k]$ be a $k$-edge coloring of a graph $G$ with color classes $E_1, \ldots, E_k$.       For  two disjoint sets $S, T\subseteq V(G)$ and color $\ell\in[k]$,   we  say that  $S$ is \dfn{$\ell$-complete}     to $T$    if  $S$ is complete to $T$  in $G_\ell$.  We simply say $S$ is \dfn{blue-complete}     to $T$  if the color $\ell$ is blue.  For convenience, we say  $v$  is   blue-complete to $T$ when $S=\{v\}$.  We say a vertex $x\in V(G)$ is \dfn{blue-adjacent} to a vertex $y\in V(G)$ if the edge $xy$ is colored blue under $\tau$; and  $x$ is \dfn{blue-complete} to an edge  $yz\in E(G)$ if $xy, xz$ are colored blue under $\tau$. Similar definitions hold when blue is replaced by another color.  Given  an $ (H_1\ldots, H_k)$-co-critical graph $G$,  we see that  $G$ admits at least one critical coloring but,  for any edge $e\in E(\overline{G})$,
$G+e$ admits  no critical coloring.  Let $\tau : E(G) \to [k]$  be a  critical coloring of $G$ such that $|E_k|$ is maximum among all critical colorings of $G$. Then $G_k$ is  $H_k$-saturated, because $G_k$ is $H_k$-free  and  $G_k+e$ has a copy of $H_k$ for each $e\in E(\overline{G_k})$.
We need  a  result of Hajnal~\cite{Hajnal} on $K_t$-saturated graphs.

\begin{thm}[Hajnal  \cite{Hajnal}]\label{t:Hajnal}
If  $G$ is   $K_t$-saturated, then  either $\Delta(G) = |G|-1$ or $\delta(G) \ge 2(t-2)$. 
\end{thm}

It is simple to check that  $r(K_s, K_t)\ge (s-1)(t-1)+1$ for all $s, t\ge2$.  This, together with a construction of Burr,  Erd\H{o}s,   Faudree and   Schelp~\cite{BEFS}, leads to the following lemma. \medskip

\begin{lem}[Burr,  Erd\H{o}s,   Faudree and   Schelp~\cite{BEFS}]\label{l:lowbd} For all $s, t\ge2$, 
\[r(K_s, K_t)\ge r(s-1, t)+2t-3\ge (s-2)(t-1)+1+2t-3=s(t-1).\]

\end{lem}

In this paper we   first  establish some   structural properties of   $(K_{t_1}, \ldots, K_{t_k})$-co-critical graphs.

\begin{lem}\label{lem:blue}
Let $G$ be a $(K_{t_1}, \ldots, K_{t_k})$-co-critical graph, where  $k\ge2$ and $ t_k\ge\cdots\ge t_1\ge3$ are integers. Let $\tau : E(G) \to [k]$ be a critical coloring of $G$ with color classes $E_1, \ldots, E_k$. Let $x\in V(G)$  and let $A_\ell:=N_\ell(x)$  for each $\ell\in[k]$. Then the following hold. 

\begin{enumerate}[(a)]
\item For each $\ell\in[k]$, $\Delta(G_\ell)\le |G|-2$ and $\omega(G_\ell[A_\ell])\le t_\ell-2$.

\item  For each $\ell\in[k]$,  every vertex in $V(G)\less N[x]$ is $\ell$-complete to a $(t_\ell-2)$-clique in  $G_\ell[A_\ell]$, and   $\omega(G_\ell[A_\ell])= t_\ell-2$.
 
\item Suppose $\tau$ is chosen such that    $|E_k|$ is maximum among all critical colorings of $G$, say the color $k$ is blue. 
\begin{enumerate}

\item[(c$_1$)]   If $A_\ell$ is blue-complete to $A_k$ for some $\ell\in[k-1]$, then    $ G_\ell[A_k]$ contains at least $t_k-2$  disjoint copies of $K_{t_\ell-1}$,  $ G_k[A_k]$ contains at least $t_\ell-1$  disjoint copies of $K_{t_k-2}$,  and so $|A_k|\ge (t_\ell-1)(t_k-2)$.

 \item[(c$_2$)] If $k=2$ and  $|A_1|=t_1-2$, then $ A_1$ is   blue-complete to $A_k$ and  $|A_2|\ge (t_1-1)(t_2-2)+1$.
\end{enumerate} 
\item Suppose $\tau$ is chosen such that    $|E_1|$ is minimum among all critical colorings of $G$.  If $k\ge3$, then $G\less E_1$ is $(K_{t_2}, \ldots, K_{t_k})$-co-critical.
 \end{enumerate}
\end{lem}

We prove \cref{lem:blue} in Section~\ref{s:blue}. Using \cref{lem:blue}, we first extend \cref{t:K3K3}  to  $(K_{t_1}, \ldots, K_{t_k})$-co-critical graphs. We  prove \cref{mindegK} here as its proof  is short.

\begin{thm}\label{mindegK}
 For all integers $k\ge2$  and  $t_k\ge \cdots\ge t_1\ge 3$, every $(K_{t_1}, \ldots,K_{t_k})$-co-critical graph has minimum degree at least $t_k-2k-1+\sum_{i=1}^k t_i$. 
\end{thm}
\begin{proof} Let $G$ be a  $(K_{t_1}, \ldots,K_{t_k})$-co-critical graph. Among all critical colorings of $G$, let $\tau:E(G)\to [k]$   be a critical coloring of $G$ with color classes $E_1, \ldots, E_k$   such that $|E_1|$ is minimum. 
We apply  induction on $k$. Assume $k=2$.    We  may assume that color 1 is red and color 2 is blue. Note that $|E_b|$ is maximum among all critical colorings of $G$ when $k=2$. By the choice of $\tau$, $G_b$ is $K_{t_2}$-saturated.  Then   $\Delta(G_b)\le |G|-2$ by  \cref{lem:blue}(a);   $\delta(G_b)\ge 2t_2-4$ by \cref{t:Hajnal}; and  $\delta(G_r)\ge t_1-2$ by \cref{lem:blue}(b). 
 Let $x\in V(G)$, and let  $A:=N_r(x)$ and $B:=N_b(x)$.    Then   $|A|\ge t_1-2$ and $|B|\ge 2t_2-4$. Furthermore,  if  $|A|= t_1-2$, then  by  \cref{lem:blue}(c$_2$),     $|B|\ge (t_1-1)(t_2-2)+1\ge 2t_2-3$ because $t_1\ge3$. 
It follows that $ d(x)=|A|+|B|\ge  2t_2+t_1-5$. 
We may assume that $k\ge3$, and the statement holds for  all $(K_{m_1}, \ldots,K_{m_{k-1}})$-co-critical graphs, where $m_{k-1}\ge \cdots\ge m_1\ge3$.   Let $G':= G\less  E_1$ and   $G'':= G\less \bigcup_{i=2}^{k}E_i$.  By \cref{lem:blue}(d),  $G'$ is $(K_{t_2}, \ldots,K_{t_k})$-co-critical.  By the induction hypothesis, $\delta(G')\ge t_k-2(k-1)-1+\sum_{i=2}^{k} t_i$. For each pair of vertices $u, v\in V(G)$ with $uv\notin E(G)$,     we see that   $u$ and $v$ share at least $t_1-2$ neighbors in common in $G''$ by \cref{lem:blue}(b) applied to $u$ and $\ell=1$. It follows that $\delta(G)\ge \delta(G')+t_1-2\ge t_k-2k-1+\sum_{i=1}^k t_i$, as desired. 
\end{proof}

We end the paper by providing  more evidence to support  \cref{question}.  We  establish  the sharp bound for the minimum degree of $(K_3, K_4)$-co-critical graphs. We  prove  \cref{K3K4} in Section~\ref{sec:K3K4}.

\begin{restatable}{thm}{threeclaw}\label{K3K4}
Every $(K_3, K_4)$-co-critical graph on $n\ge 9$ vertices has  minimum degree at least seven. The bound is sharp for all $n\ge 9$.
\end{restatable}

  \section{Proof of \cref{lem:blue}}\label{s:blue} 
 
   Let $G,  t_1, \ldots, t_k, x, A_1, \ldots, A_k$ be as given in the statement.  Let $U:=V(G)\less N[x]$. To prove \cref{lem:blue}(a), suppose there exists a vertex  $v \in V(G)$ such that $d_j(v) = |G| - 1$ for some color $j\in[k]$. We may assume that the color $j$ is red. Then there exist   $y,z \in N_j(v)$ such that $yz \notin E(G)$ because  $G$ is not a complete graph. Note that 
$G_j \less v$ is $K_{t_j-1}$-free since $G_j$ is 
$K_{t_j}$-free.   But then we obtain a critical coloring of 
$G+yz$ from $\tau$ by first coloring the edge $yz$ red, and then recoloring the edge $vy$ by a different color, say blue, in $[k]$, a contradiction.  Thus for each $\ell\in [k]$, $\Delta(G_\ell)\le |G|-2$. Since $G_\ell$ is $K_{t_\ell}$-free,   we see that    $\omega(G_\ell[A_\ell]) \le t_\ell-2$.  \medskip

To prove \cref{lem:blue}(b),   for each  $u\in U$,  we see that $G+xu$ admits no critical coloring. For each $\ell\in[k]$, by coloring the edge $xu$ with the color $\ell$, it follows that   $u$  is  $\ell$-complete to a $( t_{\ell}-2)$-clique in $G_\ell[A_\ell]$.   Therefore,  $\omega(G_\ell[A_\ell])= t_\ell-2$ by \cref{lem:blue}(a). \medskip

We  next  prove \cref{lem:blue}(c). Suppose $\tau$ is chosen such that    $|E_k|$ is maximum among all critical colorings of $G$ and the color $k$ is  blue.  To prove \cref{lem:blue}(c$_1$),        let $y\in U$.
We next show that $\omega(G_\ell[A_k])=t_\ell-1$.   Suppose $\omega(G_\ell[A_k])\le t_\ell-2$. Since  $A_\ell$ is blue-complete to $A_k$, we see that  $G+xy$ admits a  critical coloring  obtained from $\tau$ by first coloring the edge $xy$ blue, and then recoloring each edge $xv$ by color $\ell$ for all  $v\in A_k$, a contradiction.  Thus $\omega(G_\ell[A_k])=t_\ell-1$.
By \cref{lem:blue}(b), $\omega(G_k[A_k])=t_k-2$. Let $B_1, \ldots, B_p\subseteq A_k$ be pairwise disjoint $(t_k-2)$-cliques in $G_k[A_k]$ such that   $p\ge1$ is maximum; and let    $C_1, \ldots, C_q\subseteq A_k$ be pairwise disjoint $(t_\ell-1)$-cliques in $G_\ell[A_k]$ such that   $q\ge1$ is maximum. Then $G_k[B^*]$ is  $K_{t_k-2}$-free and $G_\ell[C^*]$ is  $K_{t_\ell-1}$-free, where  $B^*:=A_k\less\bigcup_{i=1}^p B_i$ and $C^*:=A_k\less\bigcup_{j=1}^q C_j$.   Note that for every $(t_k-2)$-clique $K$ in $G_k[A_k]$ with $K\notin\{B_1, \ldots, B_p\}$, we have $K\cap B_i\ne\emptyset$ for some $i\in[p]$; and for every $(t_\ell-1)$-clique $K'$ in $G_\ell[A_k]$ with $K'\notin\{C_1, \ldots, C_q\}$, we have $K'\cap C_j\ne\emptyset$ for some $j\in[q]$. Suppose $p\le t_\ell-2$ or $q\le t_k-3$.   In the first case,   $G+xy$ admits a  critical coloring  obtained from $\tau$ by first coloring the edge $xy$ blue, and then recoloring each edge $xb$ by color $\ell$ for all $b\in B\less B^*$; and in the latter case, $G+xy$ admits a  critical coloring  obtained from $\tau$ by first coloring the edge $xy$ blue, and then recoloring each edge $xc$ by color $\ell$  for all $c\in  C^*$,  a contradiction. Thus  $p\ge t_\ell-1$ and $q\ge t_k-2$, and so $|A_k|\ge (t_\ell-1)(t_k-2)$.  This proves  \cref{lem:blue}(c$_1$). \medskip

To prove \cref{lem:blue}(c$_2$), we may assume color 1 is red.  Suppose  $|A_1|=t_1-2$.  By  \cref{lem:blue}(b),  each vertex in $U$ is red-complete to $A_1$,  $G[A_1]=K_{t_1-2}$ and $G[A_1]$ has no blue edges. Hence  $A_1$ is red-complete to $U$.  Now suppose there exist $a\in A_1$ and  $b\in A_2$   such that $ab\notin E_b$. Then $ab\in E_r$    or $ab\notin E(G)$.   In the formal case, we obtain a critical coloring of $G$ from $\tau$ by recoloring the edge $ab$ blue, contrary to choice of $\tau$; in the latter case, we obtain a critical coloring of $G+ab$ from $\tau$ by  coloring  the edge $ab$ blue, a contradiction. Thus $A_1$ is blue-complete to $A_2$.   By  \cref{lem:blue}(c$_1$), $|A_2|\ge (t_1-1)(t_2-2)$. Suppose $|A_2|= (t_1-1)(t_2-2)$. Then $p=t_1-1\ge2$, where $p$ is defined in the proof of  \cref{lem:blue}(c$_1$). Since  $A_1$ is red-complete to $U$, we see that $G[U]$ has no  red   edges. Moreover, 
 $U$ is not blue-complete to $A_2$, else  
    $G[U]$ has neither red nor blue edges. But then $U$ is a stable set in $G$  and so $\chi(G)\le 1+d(x)=1+(t_1-1)(t_2-2) < r(t_1, t_2)-1$ by \cref{l:lowbd}, contrary to \cref{l:chi}.  
 Let  $z\in U$ such that $z$ is not blue-complete to $B$. We may assume that $z$  is not 
 blue-complete to $B_p$.  
 Then we obtain a critical coloring of $G+xz$ from $\tau$ by first coloring the edge $xz$ blue, and then recoloring  each $xb$ red for all $b\in B\less B_p$,  a contradiction. Thus  $|A_2|\ge (t_1-1)(t_2-2)+1$. This proves  \cref{lem:blue}(c$_2$).\medskip
  
It remains to prove \cref{lem:blue}(d). Assume $k\ge3$.  Let $G':= G\less  E_1$ and   $G'':= G\less \bigcup_{i=2}^{k}E_i$. Then $G''$  is $K_{t_1}$-free.    We next show that   $G'$ is $(K_{t_2}, \ldots,K_{t_k})$-co-critical. Note that $\tau$ restricted to $E(G')$ yields a critical coloring of $G'$. Let $e$ be an edge in the complement of $ G'$. It suffices to show that $G'+e\rightarrow (K_{t_2}, \ldots, K_{t_k})$. Suppose this is false.   Let $\sigma:E(G')\cup\{e\}\to \{2, \ldots, k\}$ be a critical coloring of $G'+e$. Let $\sigma^*$ be obtained from $\sigma$ by coloring edges in $E_1\less\{e\}$ by   color $1$. Then $\sigma^*$ is   a critical coloring of $G+e$ as   $G''$  is $K_{t_1}$-free.   Since  $G$ is  $(K_{t_1}, \ldots,K_{t_k})$-co-critical, we see that   $e\in E_1$. But then $\sigma^*$ is a critical coloring of $G$ with  fewer than $|E_1|$ edges colored by color 1,   contrary to the choice of $\tau$. This proves that  $G'$ is $(K_{t_2}, \ldots,K_{t_k})$-co-critical, as desired.\medskip
  
This completes the proof of \cref{lem:blue}.

\section{Proof of Theorem~\ref{K3K4}}\label{sec:K3K4}

In this section we prove  \cref{K3K4}, which we restate here for convenience.\threeclaw*

\begin{proof} Let $G$ be a $(K_3, K_4)$-co-critical graph on $n\ge9$ vertices. By \cref{mindegK}, $\delta(G)\ge 6$. Suppose there exists a vertex $x\in V(G)$ such that $d(x)=6$.  Among all critical colorings of $G$, let $\tau:E(G)\to$ \{red, blue\} be a critical coloring of $G$ with color classes $E_r$ and $E_b$ such that   $|E_b|$ is maximum.   By the choice of $\tau$, $G_b$ is $K_4$-saturated and $G_r$ is $K_3$-free.  Then   $\Delta(G_b)\le |G|-2$ by  \cref{lem:blue}(a);   $\delta(G_b)\ge 4$ by \cref{t:Hajnal}. Let $U:=V(G)\less N[x]$, $A:=N_r(x)$ and $B:=N_b(x)$. 
  Then $|A|+| B|=6$. Note that $|A|\ge1$ by \cref{lem:blue}(b) and $|B|\ge4$. Thus $1\le |A|\le2$ and $4\le |B|\le5$. 
 
  Let $ A:=\{a_1\}$ if $|A|=1$ and $ A:=\{a_1, a_2\}$ if $|A|=2$. Let    \[U_1:=\{v\in U\mid v \text{ is red-adjacent  to } a_1\},\]
    and let $U_2:=  U\less  U_1  $. 
  Then   $U_1=U$ if $A=\{a_1\}$ by  \cref{lem:blue}(b); and  $a_2$ is red-complete to  $  U_2$ and not red-adjacent to any vertex in $U_1$ if $A=\{a_1, a_2\}$.  Note that $G[U_i]$ has no red edges for each $i\in[2]$. Let  $B:=\{ b_1,  \ldots,  b_{|B|}\}$.   
   By \cref{lem:blue}(b),  each vertex in $U$ is blue-complete to an edge in $ G_b[B]$. Thus $ G_b[B]$ has at least one edge.   We next prove several claims.   \bigskip

\setcounter{counter}{0}

\noindent {\bf Claim\refstepcounter{counter}\label{A}  \arabic{counter}.}
 Each  vertex in $A$ is blue-complete to an edge in  $G_b[B]$.  
 
 \begin{proof} Suppose there exists a  vertex  $a\in A$ such that $a$ is not blue-complete to any edge in  $G_b[B]$. Then $G$ admits a critical coloring obtained from $\tau$ by recoloring the edge $xa$ blue, which contradicts  the minimality of $|E_r|$. 
 \end{proof}
 
\noindent {\bf Claim\refstepcounter{counter}\label{noB}  \arabic{counter}.}
 No vertex in $B$ is incident to all edges in $G_b[B]$.  
 
 \begin{proof}
 Suppose there exists  a vertex  $b\in B$ such that $G_b[B\less b]$  has no edges.  By Claim~\ref{A},  $b$ is blue-complete to $A$. Let $u\in U$. Then $G+xu$ admits  a critical coloring obtained from $\tau$ by first coloring the edge $xu$ blue and then recoloring the edge $xb$ red, a contradiction. 
  \end{proof}

  \noindent {\bf Claim\refstepcounter{counter}\label{chiU}  \arabic{counter}.}
 $\chi(G)\ge9$ and so  $\chi(G[U])\ge3$. 
 
\begin{proof} Since $d(x)=6$, we see that   $G$ is not a complete $8$-partite graph. By \cref{l:chi},  $\chi(G)\ge r(K_3,K_4)=9$.  Note that $|U|=n-7\ge 2$ and  $\chi(G)\le \chi(G[U])+d(x)$. Thus  $\chi(G[U])\ge3$.
\end{proof}

  \noindent {\bf Claim\refstepcounter{counter}\label{note}  \arabic{counter}.}
  For each $e\in E(G_b[B])$, some vertex in $ U$ is  not blue-complete to $e$.
  
  \begin{proof} Suppose there exists an edge  $e $ in $G_b[B]$ such that      every vertex in $ U$   is  blue-complete to $e$. Then $G[U]$ has no blue edges. Thus   $U_1$ and $U_2$ are stable  sets in $G$, and so  $\chi(G[U])\le2$, which contradicts  Claim~\ref{chiU}. 
  \end{proof}
  
  \noindent {\bf Claim\refstepcounter{counter}\label{UB}  \arabic{counter}.}
   $U$ is not blue-complete to $B$ and  $G_b[B]$ has at least two edges.  

 \begin{proof}  Suppose $U$ is  blue-complete to $B$. Since $G_b$ is $K_4$-free, we see that 
  $G[U]$ has no blue edges. Thus     $U_1$ and $U_2$ are stable  sets in $G$, and so   $\chi(G[U])\le2$, contrary to Claim~\ref{chiU}. This proves that  $U$ is not blue-complete to $B$. By Claim~\ref{note} and \cref{lem:blue}(b),  $ G_b[B]$ has at least one edge. 
 \end{proof}

 We first consider the case when $|A|=1$. Then  $A=\{a_1 \}$ and $B=\{b_1, \ldots, b_5\}$.    By \cref{lem:blue}(a, c$_2$), $a_1$ is red-complete to $U$ and  blue-complete to $B$. It follows that $G[U]$ has no red edges.
 We first claim that each vertex in $U$ is blue-adjacent to at least four vertices in $B$. Suppose there exists a vertex $u\in U$ such that $u$ is blue-adjacent to exactly $j$ vertices in $B$, say $b_1, \ldots,  b_j$, where $2\le j\le3$. We may assume that $b_1b_3\notin E_b$ if $j=3$ because $G_b[B]$ is $K_3$-free. But then we obtain a critical coloring of $G+xu$ from $\tau$ by first coloring the edge $xu$ blue and then recoloring the edge $xb_2$ red, a contradiction.  Thus each vertex in $U$ is blue-adjacent to at least four vertices in $B$, as claimed. 
By \cref{lem:blue}(c$_1$), $G_b[B]$   contains a matching of size two, say $\{e_1, e_2\}$.   Let $U^*$ be the set of vertices $u\in U$ such that $u$ is blue-complete to $e_1$. Then each vertex in $U\less U^*$ must be blue-complete to $e_2$ because each vertex in $U$ is blue-adjacent to at least four vertices in $B$. Since $G_b$ is $K_4$-free, we see that neither  $G[U^*]$ nor $G[U\less U^*]$ has blue edges. Thus $U^*$ and $U\less U^*$ are stable  sets in $G[U]$ and so $\chi(G[U])\le2$, contrary to Claim~\ref{chiU}.    \medskip
   
      It remains to consider the case  $|A|=2$. Then  $A=\{a_1, a_2\}$ and $B=\{b_1, \ldots, b_4\}$. Recall  that  $a_1$ is  complete to $U_1$ and anti-complete to $U_2$ in $G_r$; and $a_2$ is  complete to $U_2$ and anti-complete to $U_1$ in $G_r$.   \medskip

  \noindent {\bf Claim\refstepcounter{counter}\label{eA}  \arabic{counter}.}    For every edge $e$ in  $G_b[B]$, either $a_1$ or $a_2$   is not blue-complete to $e$. 
  
  \begin{proof}
   Suppose there exists an edge $e$ in  $G_b[B]$, say $e=b_1b_2$, such that $\{b_1, b_2\}$  is blue-complete to $A$.   Let $u\in U$ such that $u$ is  not blue-complete to $\{b_3, b_4\}$ if   $b_3b_4\in E_b$ by    Claim~\ref{note}.   We obtain a critical coloring of $G+xu$ from $\tau$ by first  coloring the edge $xu$ blue and then recoloring   edges $xb_1, xb_2$ red,   a contradiction.  
   \end{proof}
   
     For each $i\in[2]$,  by Claim~\ref{A}, there exists an edge  $e_i$ in $G_b[B]$ such that $a_i$ is blue-complete to $e_i$. Then $e_1\ne e_2$ by Claim~\ref{eA}. We  may assume that $e_1=b_1b_2$ and  $e_2=b_3b_j$ for some $j\in\{1,2,4\}$.     For each $i\in[2]$, let
  \[U^*_i:=\{v\in U_i\mid v \text{ is blue-complete to } e_1\}.\]

  \noindent {\bf Claim\refstepcounter{counter}\label{a2U2}  \arabic{counter}.}  $U_1^*$ and $\{a_1\}\cup  U_2^*$ are  stable    set in $G$. 
\begin{proof}  Recall that $a_1$ is anti-complete to 
$U_2$ in $G_r$, and $G_r[U_i]$ has no red edges for each $i\in[2]$. Since $G_b$ is $K_4$-free and $\{b_1, b_2\}$ is blue-complete to $U_1^* \cup  U_2^*\cup \{a_1\}$, we see that  $G[U_1^* \cup  U_2^*\cup \{a_1\}]$ has no blue edges. Thus $U_1^*$ and $\{a_1\}\cup  U_2^*$ are  stable    set in $G$. 
\end{proof}

  \noindent {\bf Claim\refstepcounter{counter}\label{8parts}  \arabic{counter}.}  If $a_2$ is red-complete to $e_1$, then  there exists a vertex $v$ in   $ (U_1\less U_1^*)\cup (U_2\less U_2^*) $ such that $v$ is not blue-complete to  $e_2$.
  
  \begin{proof}   Suppose $a_2$ is red-complete to $e_1$, and every  vertex  in  $ U_1\less U_1^* $ and $ U_2\less U_2^* $   is   blue-complete to $e_2$.    Then   $\{b_1, b_2\}$ is anti-complete to $U_2\less U_2^*$ in $G_r$;  similar to the proof of Claim~\ref{a2U2}, we see that $ U_1\less U_1^* $ and $ U_2\less U_2^* $ are stable    set in $G$.   Recall that every vertex in $U_2\less U_2^*$  is not blue-complete to $e_1$. Let 
  \[B_1:=\{v\in U_2\less U_2^*\mid vb_1\notin E_b\} \,\, \text{ and }\,\, B_2:=\{v\in U_2\less U_2^*\mid vb_2\notin E_b \text{ and } v\notin B_1\}.\]  
Then $B_1\cup B_2=U_2\less U_2^*$.  By Claim~\ref{a2U2},   $U_1^*$ and $\{a_1\}\cup  U_2^*$ are  stable    set in $G$.  Thus  $G$ admits a proper $8$-coloring with color classes 
  \[\{x\}\cup U_1^*, \,\, \{a_1\}\cup U_2^*, \,\,  U_1\less U_1^*,  \,\,   \{a_2\}, \,\,   \{b_1\}\cup B_1,\,\,  \{b_2\}\cup B_2,  \,\,  \{b_3\},\,\,  \{b_4\},\] 
  contrary to   Claim~\ref{chiU}.  
\end{proof}

  \noindent {\bf Claim\refstepcounter{counter}\label{three}  \arabic{counter}.}   $G_b[B]$  has at least three edges,  and so $G_b[B]=P_4$ or $G_b[B]=C_4$.

\begin{proof} 
Suppose $G_b[B]$ has exactly two edges $e_1, e_2$. If   $b_j\in\{b_1, b_2\}$,  then  $b_3b_{3-j}\notin E_b$ because $G_b[B]$ is $K_3$-free. Let $u\in U$. Then we obtain a critical coloring of $G+xu$ from $\tau$ by first  coloring the edge $xu$ blue and then recoloring the edge $ xb_j$ red,   a contradiction. Thus $e_2=b_3b_4$.  Then each vertex in $U_1\less U_1^* $  and  $U_2\less U_2^* $ is blue-complete to $e_2$.  Thus $U_1\less U_1^* $  and  $U_2\less U_2^* $ are stable sets in $G$. If $a_2b_s\notin E(G)$ for some $s\in[2]$, then  $G$ admits a proper $8$-coloring with color classes 
  \[\{x\}\cup U_1^*,\,\,  \{a_1\}\cup U_2^*, \,\,  U_1\less U_1^*,  \,\, U_2\less U_2^*,\,\,  \{a_2, b_s\},  \,\, \{b_{3-s}\}, \,\, \{b_3\}, \,\, \{b_4\},\] 
  contrary to   Claim~\ref{chiU}.  Thus  $a_2b_1, a_2b_2\in E(G)$.  By 
  Claim~\ref{8parts},  $a_2b_s\notin E_r$ for some $s\in[2]$. Then $a_2b_s\in E_b$.  By Claim~\ref{eA}, $a_2b_{3-s}\notin E_b$ and so    $a_2b_{3-s}\in E_r$.  By Claim~\ref{note} applied to $e_2$,  let $u\in U$ such that $u$ is not blue-complete to $e_2$. Then we obtain a critical coloring of $G+xu$ by first coloring the edge $xu$ blue and then recoloring $xb_s$ red, a contradiction.  This proves that  $G_b[B]$  has at least three edges.  Recall that $G_b[B]$ is $K_3$-free. By Claim~\ref{noB}, we see that  $G_b[B]=P_4$ or $G_b[B]=C_4$. 
\end{proof}

  \noindent {\bf Claim\refstepcounter{counter}\label{matching}  \arabic{counter}.}   $\{e_1, e_2\}$ is a matching in $G_b[B]$ for any choice of $e_1, e_2$.   
   
  \begin{proof}    Suppose $b_j\in\{b_1, b_2\}$. We may assume that $e_2=b_2b_3$.  Then $ b_1b_3\notin E_b$. By Claim~\ref{eA}, $a_1b_3, a_2b_1\notin E_b$. By Claim~\ref{three},  $b_1b_4\in E_b$ or $b_3b_4\in E_b$,  say the latter.  Then $b_2b_4\notin E_b$.    Moreover,   if $ub_4\notin E_b$ for some $u\in U$, then     we obtain a critical coloring of $G+xu$ from $\tau$ by first  coloring the edge $xu$ blue and then recoloring the edge  $  xb_2$ red, a contradiction.  
Thus  $b_4$ is blue-complete to $U$.   Suppose   $b_1b_4\notin E_b$.  By Claim~\ref{note}, there exists $u\in U$ such that $ub_3\notin E_b$;  we obtain a critical coloring of $G+xu$ from $\tau$ by first  coloring the edge $xu$ blue and then recoloring the edge  $  xb_2$ red, a contradiction. Thus $b_1b_4\in E_b$.  For each $i\in[2]$, let 
  \[W^*_1:=\{v\in U_1\mid vb_3\notin E_b\} \,\, \text{ and }\,\, W_2^*:=\{v\in U_2 \mid vb_1\notin E_b\}.\]  
 Then each vertex in $W_1^*$ and $U_2\less W_2^*$ is blue-complete to $\{b_1, b_4\}$, and   each vertex in $ W_2^*$ and $U_1\less W_1^*$  is blue-complete to $\{b_3, b_4\}$.  Thus $W_1^*, W_2^*, U_1\less W_1^*,  U_2\less W_2^*$ are pairwise disjoint stable  sets in $G[U]$.   
 Note that if $b_1a_2\in E_r$, then $b_1$ is anti-complete to $W_2^*$ in $G_r$. Similarly, if   $b_3a_1\in E_r$, then $b_3$ is anti-complete to $W_1^*$ in $G_r$. Thus  $G$ admits a proper $8$-coloring with color classes 
  \[\{x\}\cup W_1^*,  \,\,  W_2^*,  \,\,    U_1\less W_1^*, \,\,  U_2\less W_2^*, \,\,   \{a_2, b_1\},  \,\, \{ b_2\}, \,\, \{a_1, b_3\}, \,\, \{b_4\} \text{ if } a_2b_1, a_1b_3\notin E_r,\]
   \[\{b_3\}\cup W_1^*,  \,\,  \{x\}\cup W_2^*,  \,\,    U_1\less W_1^*, \,\,  U_2\less W_2^*, \,\,   \{a_2, b_1\},  \,\, \{ b_2\}, \,\, \{a_1\}, \,\, \{b_4\} \text{ if } a_2b_1 \notin E_r, a_1b_3\in E_r,\]
    \[ \{x\}\cup W_1^*,  \,\,  \{b_1\}\cup W_2^*,  \,\,    U_1\less W_1^*, \,\,  U_2\less W_2^*, \,\,   \{a_2\},  \,\, \{ b_2\}, \,\, \{a_1, b_3\}, \,\, \{b_4\} \text{ if } a_2b_1 \in E_r, a_1b_3\notin E_r,\]
    \[\{  b_3\}\cup W_1^*,  \,\,  \{b_1\}\cup W_2^*,  \,\,   (\{x\}\cup U_1)\less W_1^*, \,\,  U_2\less W_2^*, \,\,   \{a_2\},  \,\, \{ b_2\}, \,\, \{a_1\}, \,\, \{b_4\} \text{ if } a_2b_1,  a_1b_3\in E_r,\] 
contrary to  Claim~\ref{chiU}. \end{proof}

   By Claim~\ref{matching}, we have $e_2= b_3b_4$. By Claim~\ref{three}, we may assume that $b_2b_3\in E_b$. By Claim~\ref{matching} again, $a_1b_3, a_2b_2\notin E_b$.   Moreover,  $b_1b_3, b_2b_4\notin E_b$ because $G_b[B]$ is $K_3$-free. We first consider the case    $a_1b_4\notin E_b$ and  $a_2b_1\notin E_b$. Then $b_1b_4\in E_b$, else for any $u\in U$ we obtain a critical coloring of $G+xu$ from $\tau$ by   first  coloring the edge $xu$ blue and then recoloring  edges $xa_1, xa_2$ blue, and $xb_2, xb_3$ red, a contradiction.  Suppose there exists a vertex $u\in U$ such that $u$ is blue-complete to both $e_1$ and $e_2$. Then we obtain a critical coloring of $G+xu$ from $\tau$ by   first  coloring the edge $xu$ red and then recoloring  edges $xa_1, xa_2$ blue, and $xb_2, xb_3$ red, a contradiction. Thus no vertex in $U$ is blue-complete to both $e_1$ and $e_2$. It follows that every vertex in $U_1\less U_1^*$ and $U_2\less U_2^*$ is blue-complete to $e_2$. By Claim~\ref{8parts}, $a_2$ is not red-complete to $e_1$. We may assume that $a_2b_i\notin E(G)$ for some $i\in[2]$. Since $e_2$ is blue-complete to $U_1\less U_1^* $  and  $U_2\less U_2^* $, we see that    $U_1\less U_1^* $  and  $U_2\less U_2^* $ are stable sets in $G$. Recall that $a_2b_1, a_2b_2\notin E_b$. Then $G$ admits a proper $8$-coloring with color classes 
  \[\{x\}\cup U_1^*, \,\, U_1\less U_1^*,   \,\,  \{a_1\}\cup U_2^*, \,\, U_2\less U_2^*,  \,\,  \{a_2, b_i\}, \,\,\{b_{3-i}\}, \,\, \{b_3\}, \,\, \{b_4\},\] 
  contrary to   Claim~\ref{chiU}.

   It remains to consider the case  $a_1b_4\in E_b$ or    $a_2b_1\in E_b$.  By symmetry, we may assume that $a_1b_4\in E_b$. By Claim~\ref{matching}, $b_1b_4\notin E_b$.  
Suppose there exists a vertex $u\in U$ such that  $u$ is only blue-adjacent to $b_2, b_3$ in $B$. Then we obtain a  critical coloring of $G+xu$ from $\tau$ by first  coloring the edge $xu$ red and then recoloring  edges $xa_1, xa_2$ blue, and $xb_2, xb_3$ red, a contradiction.     Thus no vertex in $U$ is only blue-adjacent to $b_2, b_3$ in $B$.      Then each vertex in $U_1\less U_1^* $  and  $U_2\less U_2^* $ is blue-complete to $e_2$; and so $U_1\less U_1^* $  and  $U_2\less U_2^* $ are stable sets in $G$.  Recall that $a_2b_2\notin E_b$.  
 If  $a_2b_2\notin E_r$  or $a_2b_1\notin E(G)$,     then $G$ admits a proper $8$-coloring with color classes 
  \[\{x\}\cup U_1^*, \,\, \{a_1\}\cup U_2^*, \,\,  U_1\less U_1^*,  \,\, U_2\less U_2^*,  \,\, \{b_{1}\}, \,\, \{a_2, b_2\}, \,\, \{b_3\}, \,\, \{b_4\} \text{ if } a_2b_2\notin E_r, \] 
   \[\{x\}\cup U_1^*, \,\, \{a_1\}\cup U_2^*, \,\,  U_1\less U_1^*,  \,\, U_2\less U_2^*, \,\, \{a_2, b_1\},  \,\, \{b_{2}\}, \,\, \{b_3\}, \,\, \{b_4\} \text{ if } a_2b_1\notin E(G),\]  
 contrary to   Claim~\ref{chiU}.   Thus $a_2b_2\in E_r$ and $a_2b_1\in E(G)$. By 
  Claim~\ref{8parts},    $a_2b_1\in E_b$ because each vertex in $U_1\less U_1^* $  and  $U_2\less U_2^* $ is blue-complete to $e_2$. 
     By symmetry of $a_1$ and $a_2$, we see that $a_1b_3\in E_r$. 
   By Claim~\ref{note}, let $u\in U$ such that $u$ is not blue-complete to $b_2b_3$. Then $ub_2\notin E_b$ or $ub_3\notin E_b$. Thus we obtain a critical coloring of $G+xu$   obtained from $\tau$ by first coloring the edge $xu$ blue,   and then recoloring the edge $xb_1$ red  if  $ub_2\notin E_b$ and $xb_4$ red  if  $ub_3\notin E_b$, a contradiction. \medskip

This completes the proof of  \cref{K3K4}.
\end{proof}



\begin{thebibliography}{9}
 \bibitem{BEFS} S. A. Burr, P. Erd\H{o}s, R. J. Faudree and R. H. Schelp, On the difference between consecutive Ramsey numbers, Utilitas Math. 35 (1989) 115--118.
     \bibitem{P3} G. Chen, Z. Miao, Z.-X. Song and J. Zhang, On the size of special class 1 graphs and     $(P_3; k)$-co-critical graphs, Discrete Math. 344 (2021), 112604. 
    
 \bibitem{Pk} G. Chen, Z. Miao, Z.-X. Song and J. Zhang,   On the size of $(K_t, P_k)$-co-critical graphs,  Graphs and Combinatorics 38, 136  (2022).
\bibitem{c4star} G. Chen, C. Ren and Z.-X. Song, Minimizing the number of edges in $(C_4, K_{1,k})$-co-critical graphs,  Discrete Math. 348 (2025), 114409.

 
\bibitem{Chen2011}
	G. Chen, M. Ferrara, R. J. Gould, C. Magnant and J. Schmitt,  
	Saturation numbers for families of Ramsey-minimal graphs,
 J. Comb.  2 (2011) 435--455.
 
   \bibitem{DSY22}
H. Davenport, Z.-X. Song and F. Yang, On the size of $(K_t,K_{1,k})$-co-critical graphs, European J. Combin.  104 (2022), 103533.
  

 
  \bibitem{survey21} J. R. Faudree, R. J. Faudree and J. R. Schmitt, A survey of minimum saturated graphs,  Electron. J. Combin.   (2021),  DS19.
 
 
\bibitem{Ferrara2014}  
       M.  Ferrara, J. Kim and E. Yeager,  
       Ramsey-minimal saturation numbers for matchings, 
        Discrete Math. 322 (2014) 26--30.
    
\bibitem{Galluccio1992}  
       A. Galluccio,  M. Simonovits and G. Simonyi,   On the structure of co-critical graphs, 
       Graph theory, combinatorics, and algorithms, Vol. 1,2 (Kalamazoo, MI, 1992), 
       1053--1071,    Wiley-Intersci. Publ., Wiley, New York, 1995.
   
	  
 	 
\bibitem{Hajnal} 
       A. Hajnal, A theorem on $k$-saturated graphs,         Canad. J. Math. 17 (1965) 720--724.
 

\bibitem{Hanson1987}
	D. Hanson and B. Toft,  
	Edge-colored saturated graphs,
	 J. Graph Theory 
	 11 (1987) 191--196.
  	 
\bibitem{Nesetril1986}
       J. Ne$\check{s}$et$\check{r}$il,  Problem,  in  Irregularities of Partitions,   (eds G. Hal\'asz and V. T. S\'os),            Springer Verlag,  Series Algorithms and Combinatorics, vol 8, (1989) P164. (Proc. Coll. held at Fert\H{o}d, Hungary 1986). 
  
  

  \bibitem{RolekSong18}
	M. Rolek and  Z.-X. Song,
	 Saturation numbers for Ramsey-minimal graphs, 
	 Discrete Math. 341 (2018) 3310--3320.
 
 	
      

\bibitem{SongZhang21} Z.-X. Song and J. Zhang, On the size of $(K_t, \mathcal{T}_k)$-co-critical graphs,   Electron. J. Combin. 28(1) (2021), \#P1.13.
  \bibitem{Szabo1996} T. Szab\'o, \newblock On nearly regular co-critical graphs, \newblock \emph{Discrete Math.} 160 (1996) 279--281.
\end{thebibliography}
\end{document}